\documentclass[onecolumn,notitlepage]{revtex4-1}

\usepackage{graphicx}
\usepackage[utf8]{inputenc} 
\usepackage[T1]{fontenc}    
\usepackage{hyperref}       
\usepackage{url}            
\usepackage{booktabs}       
\usepackage{amsfonts}       
\usepackage{nicefrac}       
\usepackage{microtype}      
\usepackage{lipsum}
\usepackage{amsmath,amssymb,amsthm}
\usepackage[ruled]{algorithm2e}
\usepackage{mathtools}
\usepackage{physics}

\newtheorem{definition}{Definition}
\newtheorem{theorem}{Theorem}
\newtheorem{proposition}{Proposition}
\newtheorem{remark}{Remark}
\def\qed{\hfill  {\footnotesize$\blacksquare$}}

\newcommand{\relmiddle}[1]{\mathrel{}\middle#1\mathrel{}}

\newcommand{\calN}{{\mathcal N}}

\newcommand{\calU}{{\mathcal U}}
\newcommand{\calV}{{\mathcal V}}

\newcommand{\bbN}{{\mathbb N}}

\newcommand{\bbR}{{\mathbb R}}

\newcommand{\bmx}{\mbox{\boldmath $x$}}

\newcommand{\bmphi}{\mbox{\boldmath $\phi$}}
\newcommand{\bmpsi}{\mbox{\boldmath $\psi$}}

\newcommand{\tila}{{\tilde a}}
\newcommand{\tilb}{{\tilde b}}

\makeatletter  
\newcommand{\repeatable}[2]{%
    \global\@namedef{repeatable@#2}{#1}#1 \label{#2}  
}
\newcommand{\repeatref}[1]{%
    \@ifundefined{repeatable@#1}{NOT FOUND}{\footnote[0]{\eqref{#1}$ \displaystyle{ \@nameuse{repeatable@#1} } $}}%
    ~\eqref{#1}} 
\makeatother 

\begin{document}

\title{Optimal Transport-based Coverage Control for Swarm Robot Systems: Generalization of the Voronoi Tessellation-based Method}

\author{Daisuke Inoue}
\email{daisuke-inoue@mosk.tytlabs.co.jp}
\author{Yuji Ito}
\author{Hiroaki Yoshida}%
\affiliation{%
Toyota Central R\&D Labs., Inc.\\
Bunkyo-ku, Tokyo 112-0004, Japan
}%

\begin{abstract}
    Swarm robot systems, which consist of many cooperating mobile robots, have attracted attention for their environmental adaptability and fault tolerance advantages. 
    One of the most important tasks for such systems is coverage control, in which robots autonomously deploy to approximate a given spatial distribution.  
    In this study, we formulate a coverage control paradigm using the concept of optimal transport and propose a novel control technique, which we have termed the \emph{optimal transport-based coverage control (OTCC)} method. 
    The proposed OTCC, derived via the gradient flow of the cost function in the Kantorovich dual problem, is shown to covers a widely used existing control method as a special case.
    We also perform a Lyapunov stability analysis of the controlled system, and provide numerical calculations to show that the OTCC reproduces target distributions with better performance than the existing control method.
\end{abstract}

\maketitle

\section{Introduction}

Swarm robot systems, in which many mobile robots work cooperatively to perform given tasks, are expected to have strong environmental adaptability and high fault tolerance in comparison with single-robot systems~\cite{Barca2013Swarm,Brambilla2013Swarm,Rubenstein2009Scalablea}. 
One of the most fundamental and important challenges of such systems is \emph{coverage control}, in which robots move and reposition themselves autonomously so that their placement approaches a predetermined spatial distribution.
The application ranges from the optimal placement of sensor networks to efficient rescue of human life in the event of a disaster~\cite{Choset2001Coverage}. 
From the early 2000s to the present, various methods have been proposed, such as potential-function-based control~\cite{Howard2002Mobile}, probability-based control~\cite{Izumi2014Coveragea,Inoue2019Stochastic}, and broadcast-based control~\cite{Azuma2013Broadcast,Ito2020Pseudoperturbationbased}, in order to solve the coverage tasks. 
A detailed review of these methods is provided in \cite{Huang2017Review}.

Among them, the \emph{Voronoi tessellation-based coverage control (VTCC)} method proposed by Cortes et al.~\cite{Cortes2004Coverage} is a seminal work and widely used.
In this method, the coverage area is divided into subspaces referred to as Voronoi regions, each of which is assigned to a robot, and the robot is moved toward the center of gravity of its assigned Voronoi region. 
The cost function defined for the entire robot swarm is shown to decrease over time, meaning that eventually the robots are appropriately scattered throughout the coverage area.
The VTCC has been commonly used in practice for its mathematical guarantee of stability, as well as the simplicity and the scalability of the algorithm~\cite{Shibata2019Development}.

By regarding a robot swarm as an abstract group of points in Euclidean space, the coverage control is interpreted as the problem of transporting a given discrete distribution to approximate a target continuous distribution. This is commonly referred to as the \emph{optimal transport} problem, and its mathematical properties and numerical solutions have been widely investigated~\cite{Villani2003Topics, Santambrogio2010Introduction,peyre2019computational}. Recent areas of interest concerning the optimal transport problem extend to applications such as machine learning~\cite{Santambrogio2015Optimal,Courty2017Optimal,Arjovsky2017Wasserstein}, image processing~\cite{Bonneel2016Wasserstein,deGoes2012Blue}, and natural language processing~\cite{Grave2019Unsupervised}.

In this study, we formulate coverage control as an optimal transport problem in order to propose a novel control technique, which we call the \emph{optimal transport-based coverage control (OTCC)} method. 
Multi-agent control methods using optimal transport have been proposed in \cite{Bandyopadhyay2014Probabilistic,Krishnan2019Distributed}.
However, the relation between the control laws proposed in these references and the VTCC has not been investigated.
Our control method differs from existing methods in that it considers gradient flows for the cost function of the optimal transport problem,
which allows us to compare the structure and performance of the proposed OTCC with that of the VTCC.
The contributions of the present coverage control formulation is summarized as follows:
\begin{itemize}
\item The cost function for the VTCC is shown to be a special case of the cost function for the Kantorovich dual problem.
\item The new control law is derived as the gradient flows of the cost function for the Kantorovich dual problem.
\item A sufficient condition for the Lyapunov stability is provided for the controlled system, followed by a more specific condition in one-dimensional case.
\item Numerical analysis is conducted to show that the proposed method reaches a closer state to the global optimum than can be achieved via the VTCC.
\end{itemize}

\textit{Notation:}
Let $\bbR$, $\bbR_+$, and $\bbN$ be a set of real, non-negative real, and positive integer numbers, respectively.
The Euclidean norm of $x\in\bbR^n$ is represented as $\|x\|$. 
We call $\calU_z\subset \bbR^n$ is a \emph{neighborhood of $z\in\bbR^n$} if there exists an open ball $B_r\coloneqq\{x\mid \|x-z\| < r \}$ for some $r>0$ and $B_r\subset \calU_z$ holds.
For a real-valued function $G: \bbR^n\times \bbR^m\to \bbR$, we denote the partial derivative of $G(p,q)$ with respect to $p$ as $\nabla_p G$ and with respect to $q$ as $\nabla_q G$.
The higher-order derivatives follow the convention $\nabla_p^2 G = \pdv[2]{G}{p}$, $\nabla_{q}^2 G = \pdv[2]{G}{q}$, and so on.

\section{{Review of Optimal Transport}}\label{sec:OT}

This section provides a brief overview of optimal transport in order to assist in formulating the optimal transport-based coverage control. We begin by considering two continuous density functions $\rho_0:\bbR^d\to\bbR_+$ and $\rho_T:\bbR^d\to\bbR_+$ on $\bbR^d$ space, where $d\in\bbN$ is the space dimension.

\begin{definition}[Kantorovich problem]
The problem of finding a simultaneous probability density function $p(x,y)$ that minimizes the following cost function $C_\text{K}(p)$ is called the \emph{Kantorovich problem}:
\begin{align}\label{eq:Kantorovich}
C_\text{K}(p) = \int_{\bbR^d\times\bbR^d} \frac{1}{2}\|x-y\|^2 p(x,y)\dd{x} \dd{y},
\end{align}
where  $p(x,y)$ satisfies
the following conditions:
\begin{align}
\int_{\bbR^d} p(x,y)\dd{y} &= \rho_0(x),\\
\int_{\bbR^d} p(x,y)\dd{x} &= \rho_T(y).
\end{align}
\end{definition}

We denote the solution of the Kantorovich problem as $W(\rho_0,\rho_T) \coloneqq \inf_{p} C_\text{K}(p)$ and call it the \emph{Wasserstein metric}. 
In fact, $W$ is a function that measures the distance between the two distributions, and it is known that $W$ actually satisfies the axiom of distance~\cite{Santambrogio2010Introduction}.

\begin{definition}[Kantorovich dual problem]\label{prob:Kantorovich_dual}
The problem of finding integrable functions $\phi:\bbR^d\to\bbR$ and $\psi:\bbR^d\to\bbR$ that maximize the following cost functions $C_\text{KD}(\phi,\psi)$ is called the \emph{Kantorovich dual problem}:
\begin{align}\label{eq:KD}
C_\text{KD}(\phi,\psi) = \int_{\bbR^d} \phi(x) \rho_0(x) \dd{x} + \int_{\bbR^d} \psi(y) \rho_T(y) \dd{y},
\end{align}
where $\phi$ and $\psi$ satisfy the following condition:
\begin{align}
\phi(x) + \psi(y) \le \frac{1}{2}\|x-y\|^2.
\end{align}
\end{definition}

Strong duality is known to hold for the Kantorovich problem and its dual problem. 
\begin{proposition}[{\cite[Theorem 1.3]{Villani2003Topics}}]\label{prop:duality}
For the Kantorovich problem solution $p^*(x,y)$ and the dual problem solutions $\phi^*(x)$ and $\psi^*(y)$, the following equation holds:
\begin{align}\label{eq:strong_duality}
C_\text{K}(p^*) = C_\text{KD}(\phi^*,\psi^*).
\end{align}
\end{proposition}

Thus, the problem of finding the simultaneous distribution of $p$ in \eqref{eq:Kantorovich} is replaced by the problem of finding the functions $\phi$ and $\psi$ in \eqref{eq:KD}. 
In fact, it is known that we only need to find one of the two functions.

%
\begin{proposition}[{\cite[Remark 1.12]{Villani2003Topics}}]
In the Kantorovich dual problem, the following equality holds for the pair of functions $\phi^*$ and $\psi^*$ that maximize the cost function $C_{\text{KD}}(\phi,\psi)$.
\begin{align}\label{eq:c-transform}
\psi^*(y) = \inf_{x\in\bbR^d} \left\{\frac{1}{2}\|x-y\|^2 - \phi^*(x) \right\}.
\end{align}
\end{proposition}
%
In the next section we focus on the optimal transport problem where the distribution is restricted to a particular family.

\section{Proposed Control Method}\label{sec:proposed_method}
This section provides the optimal transport-based coverage control with the aid of the idea in the previous section.
We begin by considering $n\in \bbN$ mobile robots located on $\bbR^d$ space. Let the position of the $i$-th robot at time $t\in \bbR_+$ be $x_i(t)\in \bbR^d$ and let its dynamics be given as follows:
\begin{align}\label{eq:robot_dynamics}
\dot x_i(t) = u_i(t),
\end{align}
where $u_i(t)\in \bbR^d$ is the input that determines the speed of the robot. 
Next, we define a distribution formed by robots as
\begin{align}\label{eq:discrete_rho_0}
\rho(x,t) = \frac{1}{n}\sum_i^n \delta(x-x_i(t)),
\end{align}
where we define $\delta(\cdot)$ as a Dirac's delta function.
We focus on the problem of minimizing the distance (measured by $W$) between the target distribution $\rho_T$ and the distribution $\rho_0(x)=\rho(x, t)$ at each time:
%
\begin{align}\label{eq:Wasserstein}
\min_{{\small \bmx}} W(\rho(x,t),\rho_T(x)),
\end{align}
where we define $\bmx\coloneqq[x_1^\top,\ldots,x_n^\top]^\top$ as a position vector of robots. 
Our goal is to design inputs $u_i(t)\ (\forall i\in\{1,\ldots,n\})$ that achieves \eqref{eq:Wasserstein} at each time for the system of \eqref{eq:robot_dynamics}.

The results provided in the previous section yield the following expression for the problem of \eqref{eq:Wasserstein}.
%
%
\begin{proposition}\label{prop:semi-descrete}
The following equation holds:
\begin{align}\label{eq:dual_relation_discrete}
W(\rho(x,t),\rho_T(x))=\max_{{\small \bmphi}} F(\bmx,\bmphi),
\end{align}
where we define a real vector $\bmphi=[\phi_1,\ldots,\phi_n]^\top\in\bbR^n$, and we define a function $F:\bbR^{dn}\times \bbR^n\to\bbR$ as
\begin{align}
F(\bmx,\bmphi) \coloneqq \sum_{i=1}^n \left[ \left\{ \frac{1}{n} - \int_{\calV_i^\phi(x)} \rho_T(y)\dd{y} \right\} \phi_i \right. \nonumber\\
\left.+ \int_{\calV_i^\phi(x)} \frac{1}{2}\|x_i-y\|^2 \rho_T(y) \dd{y} \right].
\label{eq:dual_discrete_continuous}
\end{align}
In addition, the set $\calV_i^\phi(x)$ is called \emph{Laguerre regions}:
\begin{align}
&\calV_i^\phi(x) \coloneqq \label{eq:laguerre_set}\\
&\left\{y\in\bbR^d \relmiddle| \frac{1}{2}\|x_i-y\|^2 -\phi_i \le \frac{1}{2}\|x_j-y\|^2 -\phi_j\quad \forall j\ne i \right\}.\nonumber 
\end{align} 
%
\end{proposition}
\begin{proof}
We show that the function $W$ is transformed into \eqref{eq:dual_discrete_continuous} by using the Kantorovich duality. 
The first term in \eqref{eq:KD} is calculated as
\begin{align}
\int_{\bbR^d} \phi(x) \rho(x,t) \dd{x} &= \int_{\bbR^d} \phi(x) \frac{1}{n}\sum_i^n \delta(x-x_i) \dd{x} \nonumber\\
&= \frac{1}{n} \sum_i^n \phi_i.
\end{align}
Using \eqref{eq:c-transform}, the second term in \eqref{eq:KD} is rearranged as
\begin{align}
&\int_{\bbR^d} \psi(y) \rho_T(y) \dd{y} \nonumber\\
&= \int_{\bbR^d} \min_{i\in\{1,\ldots,n\}} \left\{ \frac{1}{2}\|x_i-y\|^2 - \phi_i \right\} \rho_T(y) \dd{y} \nonumber\\
&= \sum_{i} \int_{\calV_i^\phi(x)} \left\{ \frac{1}{2}\|x_i-y\|^2 - \phi_i \right\} \rho_T(y) \dd{y}.
\end{align}
Finally, \eqref{eq:strong_duality} ensures that \eqref{eq:dual_relation_discrete} holds.
\end{proof}

Thus, the optimization problem to be solved by the robot swarm is rearranged as the following min-max problem:
\begin{align}
\min_{{\small \bmx}} \max_{{\small \bmphi}} F(\bmx,\bmphi). \label{eq:primal_dual_problem}
\end{align}
Accordingly, we propose the following controller as a solution to the problem of \eqref{eq:primal_dual_problem}.
\begin{definition}\label{def:OT_control}
For the system in \eqref{eq:robot_dynamics}, the \emph{optimal transport-based coverage control (OTCC)} is defined as
%
\begin{align}
u_i(t) &= -k \left(x_i(t) - \tilde b_i(t)\right),\label{eq:OT_input}\\
\dot \phi_i(t) &= k'\left(\frac{1}{n} - \tilde a_i(t)\right)\label{eq:phi_dynamics},
\end{align}
where $k\in\bbR_+$ and $k'\in\bbR_+$ are design parameters that provide the feedback gain, and $\tilde a_i(t)$ and $ \tilde b_i(t)$ are defined as
\begin{align}
\tilde a_i(t) &\coloneqq \int_{\calV_i^\phi(x)} \rho_T(y) \dd{y}, \label{eq:def_tila} \\ 
\tilde b_i(t) &\coloneqq \frac{1}{\tilde a_i(t)}\int_{\calV_i^\phi(x)} y \rho_T(y) \dd{y}. \label{eq:def_tilb}
\end{align}
\end{definition}

We define and consider the following time derivative of the function $F$ in \eqref{eq:dual_discrete_continuous} along the trajectories $\bmx(t)$ and $\bmphi(t)$ of the system controlled via the OTCC:  
\begin{align}
    \dot F_{\phi}(\bmx(t)) &\coloneqq \sum_{i=1}^n \left\{\pdv{F(\bmx(t),\bmphi(t))}{x_i}\right\}^\top\dot x_i(t),\\
    \dot F_x(\bmphi(t)) &\coloneqq \sum_{i=1}^n \pdv{F(\bmx(t),\bmphi(t))}{\phi_i}\dot\phi_i(t).
\end{align}
The following proposition then justifies the conclusion that the OTCC provides a solution to the problem in \eqref{eq:primal_dual_problem}.
\begin{proposition}
The following hold for any $t\in\bbR_+$:
\begin{align}\label{eq:time_derivative_F}
    \dot F_{\phi}(\bmx(t)) \le 0,\
    \dot F_x(\bmphi(t)) \ge 0.
\end{align}
\end{proposition}
\begin{proof}
We calculate the gradients of the function $F$ as
\begin{align}
\pdv{F}{\phi_i} &= 
\frac{1}{n} - \tilde a_i(t),\label{eq:F_phi} \\
\pdv{F}{x_i} &= 
\tilde a_i(t) \left(x_i(t) - \tilde b_i(t)\right),\label{eq:F_x}
\end{align}
where we used Reynolds' transport theorem~\cite{Cortes2005Spatiallydistributed} to differentiate the functions including variables in the integration domain. 
The time derivative of $F$ along the trajectories $\bmx$ and $\bmphi$ are evaluated as
\begin{align}
    \dot F_{\phi}(\bmx(t)) &= -\frac{k}{\tila(t)} \sum_{i=1}^n \left\| \pdv{F}{x_i} \right\|^2 
    \le 0,\\
    \dot F_x(\bmphi(t)) & = k' \sum_{i=1}^n \left( \pdv{F}{\phi_i} \right)^2 \ge 0,
    \end{align}
which ends the proof.
\end{proof}

\begin{remark}
We show that the Voronoi tessellation-based coverage control (VTCC) method~\cite{Cortes2004Coverage} is regarded as a special case of the proposed OTCC.
In \cite{Cortes2004Coverage}, the following cost function is introduced:
\begin{align}\label{eq:Cortes:eval}
J(x)=\sum_{i=1}^{n} \int_{\mathcal{V}_{i}(x)} \frac{1}{2}\left\|y-x_{i}\right\|^2 \rho_T(y) \dd{y},
\end{align}
where the set $\calV_i(x)$ is called the \emph{Voronoi region}:
\begin{align}\label{eq:voronoi}
\calV_i(x) = \{y\in\bbR^d \mid \|x_i-y\| \le \|x_j-y\|\quad \forall j\ne i\}. 
\end{align}
In their study, they proposed the following control law to minimize the cost function of \eqref{eq:Cortes:eval}:
\begin{align}\label{eq:Cortes_input}
u_i(t) &= -k\left(x_i(t) - b_i(t)\right),\\
b_i(t) &\coloneqq \frac{1}{a_i(x(t))}\int_{\calV_i(t)} y \rho_T(y) \dd{y},\\
a_i(t) &\coloneqq \int_{\calV_i(x(t))} \rho_T(y) \dd{y}.
\end{align}
The OTCC in \eqref{eq:OT_input} and \eqref{eq:phi_dynamics} agree with the VTCC of \eqref{eq:Cortes_input}
when the variable $\bmphi$ is set to $\bmphi(t) \equiv 0$. 
Under this condition, the cost function for both methods (\eqref{eq:dual_discrete_continuous} and \eqref{eq:Cortes:eval}) coincide.
Therefore, the VTCC is one of the gradient flows that realize the transport from $\rho(x,t)$ to $\rho_T(x)$.
However, the VTCC is not optimal from the aspect of optimal transport because the cost function is not maximized for $\bmphi$ at each time. In contrast, the proposed OTCC overcomes this problem and is expected to provide better control performance.
\end{remark}
\begin{remark}
    We discuss the difference between the Voronoi region in \eqref{eq:voronoi} and the Laguerre region in \eqref{eq:laguerre_set}.
    In both sets, the boundary is perpendicular to the line between neighboring robots $i$ and $j$. 
    While the boundary in the Voronoi region is a bisector, it is not a bisector in the Laguerre region if $\bmphi$ is not zero.
    Specifically, when $\phi_i$ is larger than $\phi_j$, the boundary moves towards the robot $j$, and consequently robot $i$ acquires more area. 
\end{remark}

\begin{remark}
    Under suitable conditions, the proposed algorithm is scalable in the sense that each robot only needs the information of neighboring robots. 
    Indeed, the controllers in \eqref{eq:OT_input} and \eqref{eq:phi_dynamics} are computed using $x_{j}$ and $\phi_{j}$ of the robots $j \in \calN_{i}^{\phi}$, where $\calN^\phi_i\coloneqq \{j \mid \calV^\phi_i\cap\calV^\phi_j\ne\emptyset\}$ is the set of adjacent robots in the Laguerre sense. 
%
    Hence, we focus on clarifying the condition that each robot is capable of obtaining the information of the neighboring robots.
    Suppose that each robot $i$ has a measurement range $R_i$ and obtains the information of robots $j$ satisfying $\|x_j-x_i\|\le R_i$. 
    Then, the proposed algorithm is feasible if $R_i$ satisfies $\|x_j-x_i\|\le R_i$ for any $i$ and $j \in \calN_{i}^{\phi}$. 
    Such a range $R_i$ does not increase in a common situation where many robots are distributed on the workspace and $|\phi_{i}-\phi_{j}|$ (for any $i,j$) is sufficiently smaller than the size of the workspace.
\end{remark}

%
%
\section{Stability Analysis}\label{sec:stability_proof}

In this section, we prove the Lyapunov stability of the system's equilibrium when using our proposed OTCC. 
Suppose that the dynamical system represented by \eqref{eq:robot_dynamics} are controlled with the OTCC in \eqref{eq:OT_input} and \eqref{eq:phi_dynamics}. 
The system equilibrium $(\bmx^*,\bmphi^*)$ is then characterized as a point that satisfies the following conditions:
\begin{align}
\begin{split}\label{eq:equilibrium}
\begin{cases}
x_i^* &= \tilb_i,\\
\phi_i^* &\in\left\{\phi\relmiddle{|} \tila_i = \frac{1}{n}\right\},
\end{cases}
& \quad\text{for } i\in\{1,\ldots,n\},
\end{split}
\end{align}
where we assume that $x_i^*\ne x_j^*$ holds for all $i\ne j$.

\begin{theorem}\label{thm:convergence_nd}
    The equilibrium point $(\bmx^*, \bmphi^*)$ is Lyapunov stable if there exist neighborhoods $\calU_{x^*}$ around $\bmx^*$ and for any $\bmx\in \calU_{x^*}$,
    \begin{align}\label{eq:thm_cond_nd}
        H(\bmx,\bmphi^*) \text{ is positive definite,}
    \end{align}
    where $H(\bmx,\bmphi)\in\bbR^{nd\times nd}$ is defined as 
    \begin{align}
            &\left(H(\bmx,\bmphi)\right)_{ij}\ (\in \bbR^{d\times d})=\label{eq:ddfx}\\
            &\begin{dcases}
                \frac{1}{\|x_j-x_i\|} \int_{\partial \calV^\phi_{ij}} (x_i-s)(x_j-s)^\top \rho_T(s) \dd{s},
                \quad \qq*{$j \neq i$,}\\
                \tila_i I_d - \sum_{\ell\neq i} \frac{1}{\|x_\ell-x_i\|} \int_{\partial \calV^\phi_{i\ell}} (x_i-s)(x_\ell-s)^\top \rho_T(s) \dd{s},\\
                \omit \hfill \qq*{$j = i$.}
            \end{dcases}\nonumber
    \end{align}
    Here, $\left(H(\bmx,\bmphi)\right)_{ij}$ denotes $(i,j)$-block matrix element of $H(\bmx,\bmphi)$, and $I_d\in\bbR^{d\times d}$ denotes $d$-dimensional identity matrix.
    We define $\partial \calV^\phi_{ij} \coloneqq \calV_i^\phi(x)\cap \calV_j^\phi(x)$ as the Laguerre region boundary of $i$ and $j$, and the integral in \eqref{eq:ddfx} is $0$ when $\partial \calV^\phi_{ij}$ is an empty set. 
\end{theorem}

We prove Theorem~\ref{thm:convergence_nd} by utilizing following proposition.
%
%
\begin{proposition}[\cite{Cherukuri2017Saddlepoint}]\label{prop:saddle-stability}
For a continuous and differentiable function $G:\bbR^n\times\bbR^n\to\bbR$, consider a system with the following dynamics:
\begin{align}
\begin{split}\label{eq:double_gradient_flows}
\dot p(t) &= - \nabla_p G(p,q),\\
\dot q(t) &= \nabla_q G(p,q).
\end{split}
\end{align}
If $G$ is convex-concave around the saddle point $(p^*,q^*)$, then the system of \eqref{eq:double_gradient_flows} is Lyapunov stable at $(p^*,q^*)$.
\end{proposition}

Here, the convex-concave function and the saddle point are defined below.
\begin{definition}[Convex-concave function]
A function $g:\bbR^n\to\bbR$ is \emph{convex} around $\bar z\in\bbR^n$ if a neighborhood $\calU_{\bar z}$ exists and the inequality of $\lambda g(x)+(1-\lambda) g(y) \ge g(\lambda x+(1-\lambda) y)$ holds for any $x,y\in\calU_{\bar z}$ and $\lambda\in[0,1]$.
The function $g$ is \emph{concave} around $\bar z$ if the inverse inequality holds. Furthermore, the function $G:\bbR^n\times\bbR^n\to\bbR$ is \emph{convex-concave} around $(\bar p,\bar q)$ if $p\mapsto G(p,\bar q)$ is convex around $\bar p$ and $q\mapsto G(\bar p,q)$ is concave around $\bar q$.
\end{definition}

%
\begin{definition}[Saddle point]
For a continuous and differentiable function $G:\bbR^n\times\bbR^n\to\bbR$, a point $(p^*,q^*)\in\bbR^n\times\bbR^n$ is a \emph{saddle point} if there exists neighbors $\calU_{p^*}$ and $\calU_{q^*}$, and the following holds for any $p\in\calU_{p^*}$ and $q\in\calU_{q^*}$:
\begin{align}\label{eq:def_saddle}
G(p^*, q)\le G(p^*, q^*)\le G(p, q^*).
\end{align}
\end{definition}

Proposition~\ref{prop:saddle-stability} ensures that it is sufficient to show that the function $F$ in \eqref{eq:dual_discrete_continuous} is convex-concave around the equilibrium $(\bmx^*,\bmphi^*)$ of \eqref{eq:equilibrium}, and that $(\bmx^*,\bmphi^*)$ is the saddle point of $F$. The former is guaranteed with the aid of the following two propositions:

\begin{proposition}\label{prop:phi_concave}
    The function $\phi\mapsto F(\bmx^*, \bmphi)$ with $\bmx$ fixed to $\bmx^*$ is concave around $\bmphi^*$.
\end{proposition}
\begin{proof}
    We denote the above function as $F_{x^*}(\bmphi)$. In order to use the second-order sufficient condition for the concave function~\cite{Boyd2004Convex}, we evaluate $\nabla^2_\phi F_{x^*}(\bmphi)$, the second-order derivative of $F_{x^*}(\bmphi)$ for $\bmphi$. Using Reynolds' transport theorem, it is shown that
    \begin{align}\label{eq:ddfphi}
        &\left(\nabla^2_\phi F_{x^*}(\bmphi)\right)_{ij}\\
        &=\begin{dcases}
        \frac{1}{\|x_j^*-x_i^*\|} \int_{\partial \calV^\phi_{ij}} \rho_T(s) \dd{s}, & \qq*{for $j \neq i$,}\\
        - \sum_{\ell\neq i} \frac{1}{\|x_\ell^*-x_i^*\|} \int_{\partial \calV^\phi_{i\ell}} \rho_T(s) \dd{s}, & \qq*{for $j = i$.}
        \end{dcases}\nonumber
    \end{align}
    Equation~\eqref{eq:ddfphi} is a weighted graph Laplacian multiplied by $-1$. Thus, $\bmpsi^\top \nabla^2_{\phi} F_{x^*}(\bmphi) \bmpsi \le 0$ holds for any $\bmphi,\bmpsi\in\bbR^n$,
    which means that $\nabla^2_\phi F_{x^*}(\bmphi)$ is negative semidefinite at any point $\bmphi\in\bbR^n$. 
\end{proof}

\begin{proposition}\label{prop:x_convex_nd}
    If the condition of \eqref{eq:thm_cond_nd} holds, the function $x\mapsto F(\bmx, \bmphi^*)$ with $\bmphi$ fixed to $\bmphi^*$ is convex around $\bmx^*$.
\end{proposition}
\begin{proof}
    We denote the above function as $F_{\phi^*}(\bmx)$.
    Using Reynolds' transport theorem, the second-order derivative of $F_{\phi^*}$ is calculated as $\nabla^2_x F_{\phi^*}(\bmx) = H(\bmx,\bmphi^*)$ defined in \eqref{eq:ddfx}. 
    Thus, $F_{\phi^*}$ is convex around $\bmx^*$ if \eqref{eq:thm_cond_nd} holds.
\end{proof}

The equilibrium is directly shown to be the saddle point using the convex-concavity.
\begin{proposition}\label{prop:saddle-point}
If the condition of \eqref{eq:thm_cond_nd} holds, $(\bmx^*, \bmphi^*)$ is the saddle point of the function $F$ of \eqref{eq:dual_discrete_continuous}.
\end{proposition}
\begin{proof}
We show that \eqref{eq:def_saddle} holds. By using Taylor's formula, the following holds for the equilibrium point $(\bmx^*, \bmphi^*)$ and any $\bmphi$ in the neighborhood of $\bmphi^*$:
\begin{align}\label{eq:Taylor}
& F(\bmx^*,\bmphi) - F(\bmx^*,\bmphi^*) 
= \nabla_\phi F_{x^*}(\bmphi^*)^\top(\bmphi-\bmphi^*)\nonumber\\
&\quad + (\bmphi-\bmphi^*)^\top\frac{1}{2}\nabla^2_\phi F_{x^*}(\tilde \bmphi)(\bmphi-\bmphi^*)\le 0,
\end{align}
where $\tilde \bmphi$ is a point between $\bmphi^*$ and $\bmphi$. 
The inequality comes from the definition of the equilibrium point in \eqref{eq:equilibrium} and the concavity of $F_{x^*}(\bmphi^*)$.
The other inequality in \eqref{eq:def_saddle} is also shown by considering Taylor's formula in the neighborhood of $\bmx^*$.
\end{proof}

Finally, Theorem~\ref{thm:convergence_nd} is proved by using above propositions.

%
%
\quad \textit{Proof of Theorem~\ref{thm:convergence_nd}:}
Propositions \ref{prop:phi_concave} and \ref{prop:x_convex_nd} show that the function $F$ of \eqref{eq:dual_discrete_continuous} is convex-concave around $(\bmx^*, \bmphi^*)$ when \eqref{eq:thm_cond_nd} is satisfied. 
In addition, Proposition~\ref{prop:saddle-point} ensures that $(\bmx^*, \bmphi^*)$ is the saddle-point of $F$. 
As a result, Proposition~\ref{prop:saddle-stability} guarantees that $(\bmx^*, \bmphi^*)$ is Lyapunov stable.
\qed
\vspace{2mm}

By limiting the space dimension to $d=1$, we reduce the condition in Theorem \ref{thm:convergence_nd} to the following explicit form.
\begin{theorem}\label{thm:convergence}
Suppose that $d = 1$ holds for the space dimension.
Then, $(\bmx^*, \bmphi^*)$ is Lyapunov stable provided that the following inequality holds for any $i\in\{1,\ldots,n\}$:
\begin{align}\label{eq:thm_cond}
h_i(\bmx^*,\bmphi^*) < 1,
\end{align}
where the function $h_i:\bbR^n\times\bbR^n\to\bbR$ is defined as 
\begin{align}
&h_i(\bmx,\bmphi)\label{eq:thm_cond_1}\\
&\coloneqq 2n \sum_{j\in\calN^\phi_i} \left|\frac{x_{ij}}{2}+\frac{\phi_{ij}}{x_{ij}}\right| \left(\frac{1}{2} + \frac{|\phi_{ij}|}{x_{ij}^2}\right)\rho_T\left(\frac{x_{ij}^+}{2}-\frac{\phi_{ij}}{x_{ij}}\right).\nonumber
\end{align}
In \eqref{eq:thm_cond_1}, we define $x_{ij}\coloneqq(x_i-x_j)/2$, $x_{ij}^+\coloneqq(x_i+x_j)/2$, and $\phi_{ij}\coloneqq(\phi_i-\phi_j)/2$.
\end{theorem}


The proof of Theorem \ref{thm:convergence} is obvious through the following proposition.

\begin{proposition}\label{prop:x_convex}
Suppose that $d = 1$ holds. 
If the inequality of \eqref{eq:thm_cond} holds for any $i\in\{1,\ldots,n\}$, the function $x\mapsto F(\bmx, \bmphi^*)$ with $\bmphi$ fixed to $\bmphi^*$ is convex around $\bmx^*$.
\end{proposition}
\begin{proof}
Using $d=1$, we assume $x_i<x_{i'}$ for $i<i'$ without a loss of generality. Then, $\partial \calV^\phi_{ij}$ exists as $\partial \calV^\phi_{ij} = \{x_{ij}^+/2 - \phi_{ij}/x_{ij}\}$ only if $j\in\{i\pm 1\}$, and $\partial \calV^\phi_{ij}=\emptyset$ otherwise. 
The value of $\nabla^2_x F_{\phi^*}(\bmx)$ in \eqref{eq:ddfx} is then calculated as follows:
\begin{align}\label{eq:ddfx_1}
&\left(\nabla^2_x F_{\phi^*}(\bmx)\right)_{ij}=\\
&\begin{dcases}
-\left(\frac{x_{ij}}{2}+\frac{\phi_{ij}}{x_{ij}}\right)\left(\frac{1}{2} - \frac{\phi_{ij}}{x_{ij}^2}\right)\rho_T\left(\frac{x_{ij}^+}{2}-\frac{\phi_{ij}}{x_{ij}}\right),&\qq*{$j=k$,}\\
\left(\frac{x_{ij}}{2}+\frac{\phi_{ij}}{x_{ij}}\right)\left(\frac{1}{2} - \frac{\phi_{ij}}{x_{ij}^2}\right)\rho_T\left(\frac{x_{ij}^+}{2}-\frac{\phi_{ij}}{x_{ij}}\right),&\qq*{$j=h$,}
\end{dcases}\nonumber
\end{align}
where we denote $k\coloneqq i-1$ and $h\coloneqq i+1$.
For $j=i$, 
\begin{align}
&\left(\nabla^2_x F_{\phi^*}(\bmx)\right)_{ii}=\tila_i\nonumber\\
&+ \left(\frac{x_{ih}}{2}+\frac{\phi_{ih}}{x_{ih}}\right)\left(\frac{1}{2} + \frac{\phi_{ih}}{x_{ih}^2}\right)\rho_T\left(\frac{x_{ih}^+}{2}-\frac{\phi_{ih}}{x_{ih}}\right)\nonumber\\
&-\left(\frac{x_{ik}}{2}+\frac{\phi_{ik}}{x_{ik}}\right)\left(\frac{1}{2} + \frac{\phi_{ik}}{x_{ik}^2}\right)\rho_T\left(\frac{x_{ik}^+}{2}-\frac{\phi_{ik}}{x_{ik}}\right).
\end{align}
By using Gershgorin theorem~\cite{Horn2012Matrix}, all eigenvalues of $\nabla^2_x F_{\phi^*}(\bmx)$ lie within a disk with a center of $(\nabla^2_x F_{\phi^*}(\bmx))_{ii}$ and a radius of $\sum_{j}\left|(\nabla^2_x F_{\phi^*}(\bmx))_{ij}\right|$, so that the following is obtained as a sufficient condition for all eigenvalues to be positive:
\begin{align}\label{eq:condition_1}
    \tila_i
    > \sum_{j\in \calN_i^\phi} 
    2\left| \frac{x_{ij}}{2}+\frac{\phi_{ij}}{x_{ij}}\right| \left(\frac{1}{2} + \frac{|\phi_{ij}|}{x_{ij}^2}\right)\rho_T\left(\frac{x_{ij}^+}{2}-\frac{\phi_{ij}}{x_{ij}}\right).
\end{align}
Therefore, in order for $\nabla^2_x F_{\phi^*}(\bmx)$ to be positive definite at the equilibrium $\bmx^*$, it is sufficient that \eqref{eq:thm_cond} holds, where we use the fact that $\tila_i=1/n$ holds at the equilibrium.
From the assumption that $x_i^*\ne x_j^*\ (\forall i\ne j)$, we see that both sides of \eqref{eq:condition_1} are continuous functions. Thus, if \eqref{eq:condition_1} holds at the equilibrium $\bmx^*$, it remains in its neighborhood, which shows that the function $F$ is convex around $\bmx^*$.
\end{proof}


%
%
\section{Numerical Experiments}\label{sec:numerical_example}

\begin{figure}[t]
    \centering
    \includegraphics[width=100mm,bb=0 0 400 300]{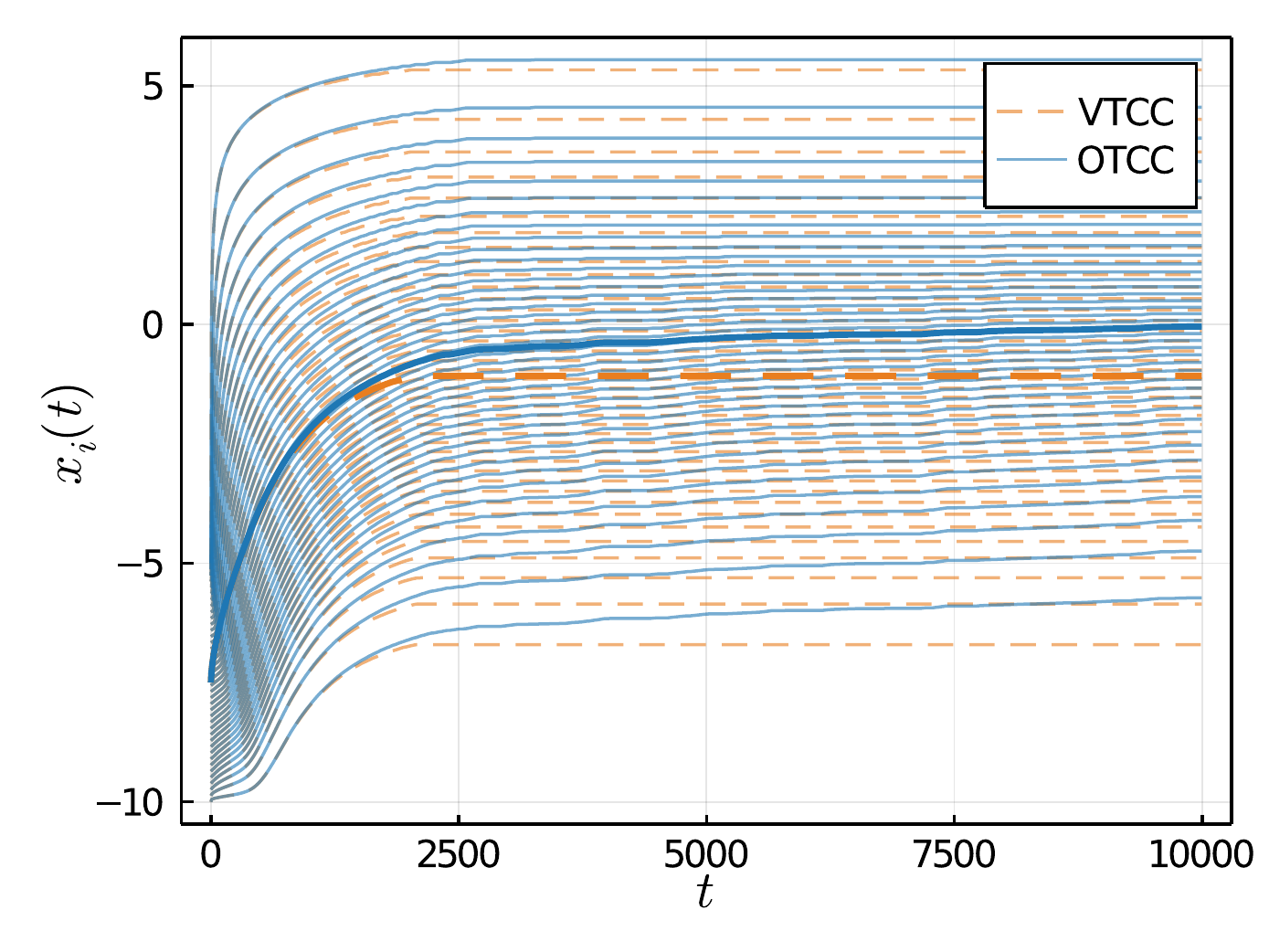}
    \caption{
    Time evolution of the position of each robot. Here, the proposed OTCC (blue-solid) is compared with the VTCC in \cite{Cortes2004Coverage} (orange-dashed).
    The bold lines represent the average value of all robot positions.
    }
    \label{fig:trajectory_VT}
\end{figure}

In this section, numerical analysis is conducted to verify the performance of the proposed OTCC. 
We first consider a one-dimensional case with 40 robots uniformly distributed over the interval $[-10,-5]$.
The target density distribution $\rho_T$ is set as $\rho_T(x)=N(0,3)$, where $N(\mu,\sigma^2)$ denotes the density function for the normal distribution with a mean $\mu$ and a variance $\sigma^2$.
We use $k = 0.5$ for the feedback gain in the VTCC, and $k = 0.5$ and $k' = 1.0\times 10^{-4}$ as the feedback gains in the proposed OTCC. 
The integrations in \eqref{eq:def_tila} and \eqref{eq:def_tilb} are carried out numerically by restricting the robot workspace to the interval of $[-10,10]$ and discretizing it into small cells.
Each robot then determines the ownership of each cell based on the definition of the Laguerre regions and performs numerical integration over the area belonging to the robot.
The two control methods are performed for the above system and the trajectory of the robots' position is shown in Fig.~\ref{fig:trajectory_VT}.
In both methods, the robots that were centered around $x = -7.5$ at time $t = 0$ move over time to present the desired distribution $\rho_T$ centered at $x = 0$. 
However, in the VTCC, many robots remain in the region of $x < 0$, and the mean value of the distribution does not approach to $x = 0$. 
In contrast, the OTCC reproduces the shape of the target distribution $\rho_T$ more closely over time.
For more quantitative analysis, we examine the value of the cost function in \eqref{eq:Cortes:eval} in the steady state. 
The value of the cost function in the VTCC is $1.36\times 10^{-3}$, whereas the value in the proposed OTCC is $1.08\times 10^{-3}$, which suggests that the latter provides better control.
To check whether the conditions of Theorem~\ref{thm:convergence} are satisfied, we next examine the value of the left-hand side of \eqref{eq:thm_cond} in the steady state. The maximum value for all robots is 
$5.69\times 10^{-1}\ (<1)$, which indicates that the stationary point is Lyapunov stable.

\begin{figure}[t]
\centering
\includegraphics[width=100mm,bb=0 0 595 842]{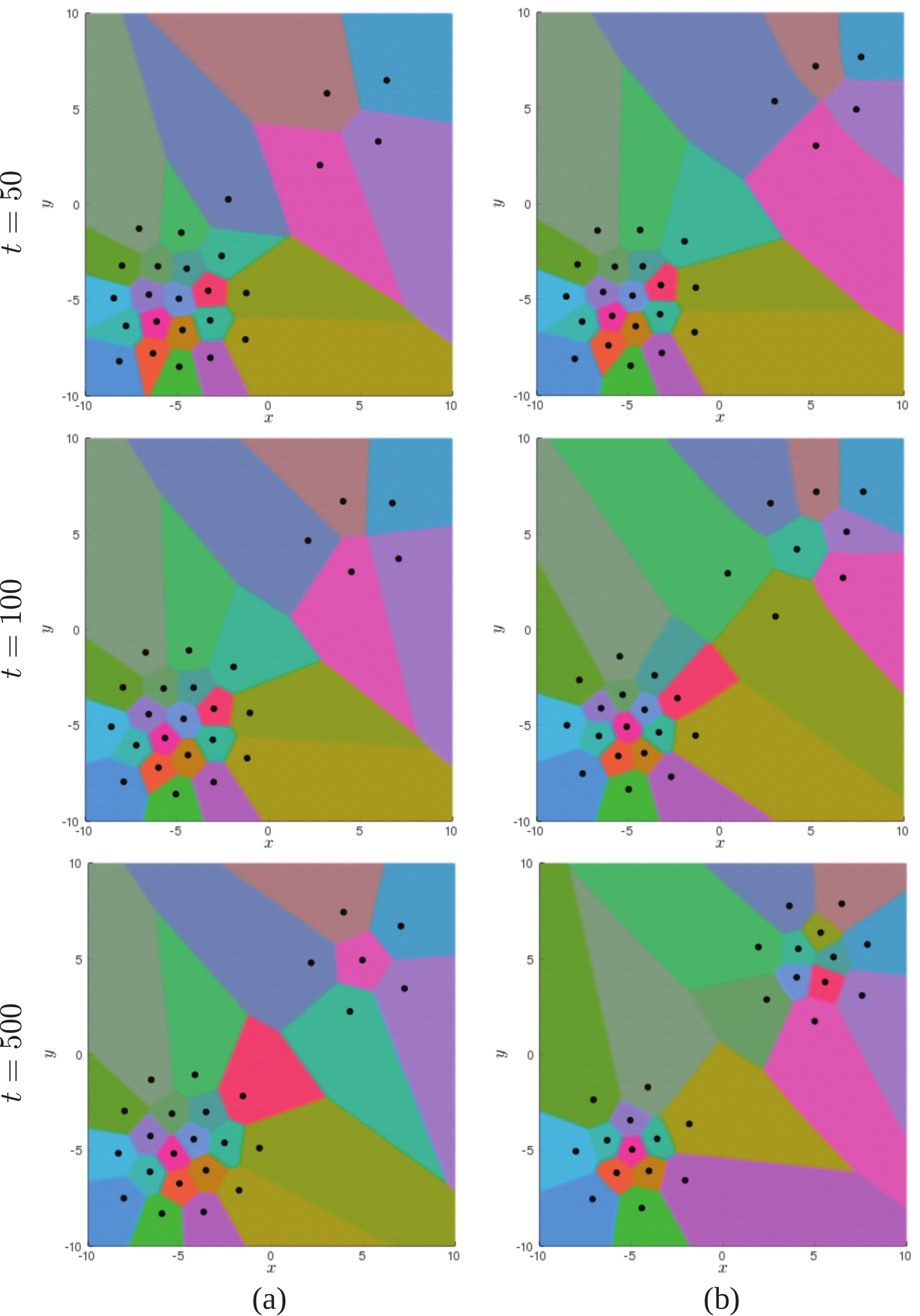}
\caption{
Time evolution of the position of each robot.
(a) The VTCC in \cite{Cortes2004Coverage}, (b) The proposed OTCC.
}
\label{fig:trajectory}
\end{figure}
Next, we consider a two-dimensional case, with 25 robots uniformly distributed over the interval of $[-8,-2]^2$ and with an interval of $[-10,10]^2$ for robot workspace.
%
We use $\rho_T(x) = \frac{1}{2}\left\{ N(\mu_1,\Sigma) + N(\mu_2,\Sigma) \right\}$ as the target density distribution, while $\mu_1 = [-5,-5]^\top,\ \mu_2 = [5,5]^\top$, and $\Sigma=\text{diag}(4,4)$ are used for the mean and covariance.
We set $k = 0.5$ for the feedback gain in the VTCC, and $k = 0.5$ and $k' = 5.0\times 10^{-2}$ for the feedback gains in the proposed OTCC. 
Fig.~\ref{fig:trajectory} shows the visualized trajectory of the positions of the robots using two control laws. The black circles represent the positions of the robots, and the areas painted in different colors represent the Voronoi/Laguerre region to which each robot belongs. In both methods, the robots move in a way that reproduces the desired distribution $\rho_T$. 
%
In the VTCC, only 6 of the 25 robots move to the upper-right distribution, while the remaining robots stay in the lower-left distribution.
In contrast, in the OTCC, 12 of the 25 robots move to the upper-right distribution, thus suggesting that the shape of the target distribution $\rho_T$ is better reproduced.
We calculate the value of the cost function of \eqref{eq:Cortes:eval} at the steady state. The value in the VTCC is $9.70\times 10^{-1}$, whereas the value in the proposed OTCC is $8.84\times 10^{-1}$, thus indicating the distribution is better reproduced in the proposed method.

\begin{remark}
    We discuss the reasons of higher performance of the OTCC.
    This is because that the equilibrium condition in the OTCC is stricter than that in the VTCC, which makes the robot less likely to be trapped in the stationary point.
    The equilibrium condition in the VTCC is that the robot is placed at the center of gravity of each region $(x_i=b_i)$, and once this condition is satisfied, each robot does not move thereafter. 
    In the OTCC, however, there is an additional condition that the weighted area of each region is equal ($\tila_i= 1/n$ in \eqref{eq:equilibrium}). 
    Thus, even if the former condition is satisfied, robots continue to move unless the latter condition is satisfied, which is why robots are less likely to be trapped at the stationary point.
\end{remark}

\section{Conclusion}\label{sec:conclusion}

In this paper, we propose the optimal transport-based coverage control (OTCC) method as an improvement to the Voronoi tessellation-based coverage control (VTCC) method. 
Our proposed method, which is derived via the Kantorovich dual problem, is consistent with the VTCC with setting $\bmphi(t)\equiv 0$.
This correspondence successfully reconsiders the VTCC in the optimal transport framework. 
We also derive the conditions of Lyapunov stability for the controlled system in Theorem~\ref{thm:convergence_nd} and \ref{thm:convergence}, showing that once the robot swarm reaches the equilibrium point, they remain in the neighborhood. Numerical calculations clearly show that our proposed method is more capable of escaping the local optimum point and achieving better control than the VTCC.
%
As one of our next research topics, we plan to extend the applicability of the OTCC to more general class of systems.



\end{document}